\documentclass[article]{IEEEtran}

\usepackage{amsmath,amssymb,amsfonts}%
\usepackage{color}%
\usepackage{theorem}
\usepackage{dsfont}%
\usepackage{hyperref}%

\theoremstyle{change}%

\newtheorem{definition}{Definition:}[section]%
\newtheorem{theorem}[definition]{Theorem:}%
\newtheorem{lemma}[definition]{Lemma:}%
\newtheorem{corollary}[definition]{Corollary:}%
{\theorembodyfont{\rmfamily} \newtheorem{remark}[definition]{Remark:}}%
{\theorembodyfont{\rmfamily} \newtheorem{example}[definition]{Example:}}%

\newcommand{\CC}{\mathcal{C}}%
\newcommand{\EC}{\mathcal{E}}%
\newcommand{\FC}{\mathcal{F}}%
\newcommand{\MC}{\mathcal{M}}%
\newcommand{\PC}{\mathcal{P}}%
\newcommand{\QC}{\mathcal{Q}}%
\newcommand{\RC}{\mathcal{R}}%
\newcommand{\T}{\mathbb{T}}%
\newcommand{\N}{\mathbb{N}}%
\newcommand{\Z}{\mathbb{Z}}%
\newcommand{\R}{\mathbb{R}}%
\newcommand{\ep}{\varepsilon}%
\newcommand{\tm}{\times}%
\newcommand{\rmd}{\mathrm{d}}%
\newcommand{\rmD}{\mathrm{D}}%
\newcommand{\rme}{\mathrm{e}}%
\newcommand{\tp}{\mathrm{top}}%
\newcommand{\spn}{\mathrm{span}}%
\newcommand{\sep}{\mathrm{sep}}%
\newcommand{\diam}{\mathrm{diam}}%
\newcommand{\supp}{\mathrm{supp}}%
\newcommand{\rmS}{\mathrm{S}}%

\begin{document}

\title{On optimal coding of non-linear dynamical systems}

\author{Christoph Kawan and Serdar Y\"{u}ksel\thanks{C.~Kawan is with the Faculty of Computer Science and Mathematics, University of Passau, 94032 Passau, Germany (e-mail: christoph.kawan@uni-passau.de). S.~Y\"uksel is with the Department of Mathematics and Statistics, Queen's University, Kingston, Ontario, Canada, K7L 3N6 (e-mail: yuksel@mast.queensu.ca).}}%
\date{}%
\maketitle%

\begin{abstract}
We consider the problem of zero-delay coding of a dynamical system over a discrete noiseless channel under three estimation criteria concerned with the low-distortion regime. For these three criteria, formulated stochastically in terms of a probability distribution for the initial state, we characterize the smallest channel capacities above which the estimation objectives can be achieved. The results establish further connections between topological and metric entropy of dynamical systems and information theory.%
\end{abstract}




\section{Introduction}%

In this paper, we develop further connections between the ergodic theory of dynamical systems and information theory, in the context of zero-delay coding and state estimation for non-linear dynamical systems, under three different criteria. In the following, we first introduce the problems considered in the paper, and present a comprehensive literature review.

\subsection{The Problem}

Let $f:X\rightarrow X$ be a homeomorphism (a continuous map with a continuous inverse) on a compact metric space $(X,d)$. We consider the dynamical system generated by $f$, i.e.,%
\begin{equation*}
  x_{t+1} = f(x_t),\quad x_0 \in X,\ t \in \Z.%
\end{equation*}
The space $X$ is endowed with a Borel probability measure $\pi_0$ which describes the uncertainty in the initial state.%

Suppose that a sensor measures the state at discrete sampling times, and a coder translates the measurements into symbols from a finite coding alphabet $\MC$ and sends this information over a noiseless digital channel to an estimator. The input signal sent at time $t$ through the channel is denoted by $q_t$. At the other end of the channel, the estimator generates an estimate $\hat{x}_t \in X$ of $x_t$, using the information it has received through the channel up to time $t$. We consider arbitrary (but causal / zero-delay) coding and estimation policies. That is, $q_t$ is generated by a (measurable) map%
\begin{equation}\label{eq_coder_map}
  q_t = \delta_t(x_0,x_1,\ldots,x_t),\quad \delta_t:X^{t+1} \rightarrow \MC,%
\end{equation}
and the estimator output at time $t$ is given by a map%
\begin{equation}\label{eq_estimator_map}
  \hat{x}_t = \gamma_t(q_0,q_1,\ldots,q_t),\quad \gamma_t:\MC^{t+1} \rightarrow X.%
\end{equation}
Throughout the paper, we assume that the channel is noiseless with finite input alphabet $\MC$. Hence, its capacity is given by%
\begin{equation*}
  C = \log_2|\MC|.%
\end{equation*}

Our aim is to characterize the smallest capacity $C_0$ above which one of the following estimation objectives can be achieved for every $\ep>0$.%

\begin{enumerate}
\item[(E1)] Eventual almost sure estimation: There exists $T = T(\ep) \in \N$ s.t.%
\begin{equation*}
  P(d(x_t,\hat{x}_t) \leq \ep) = 1 \mbox{\quad for all\ } t \geq T.%
\end{equation*}
\item[(E2)] Asymptotic almost sure estimation:%
\begin{equation*}
  P\bigl(\limsup_{t\rightarrow\infty} d(x_t,\hat{x}_t) \leq \ep\bigr) = 1.%
\end{equation*}
\item[(E3)] Asymptotic estimation in expectation:%
\begin{equation*}
  \limsup_{t\rightarrow\infty}E[d(x_t,\hat{x}_t)^p] \leq \ep, \quad p > 0.%
\end{equation*}
\end{enumerate}

The smallest channel capacity above which one of these objectives can be achieved for a fixed $\ep>0$ will be denoted by $C_{\ep}$. We are particularly interested in%
\begin{equation*}
  C_0 := \lim_{\ep\downarrow0}C_{\ep}.%
\end{equation*}

In most of our results we will additionally assume that $\pi_0$ is $f$-invariant, i.e., $\pi_0 = \pi_0 \circ f^{-1}$. Sometimes we will need the stronger assumption that $\pi_0$ is ergodic.%

Given an encoding and decoding policy, the estimate $\hat{x}_t$ at time $t$ is determined by the initial state $x_0$, e.g., $\hat{x}_1 = \gamma_1(\delta_0(x_0),\delta_1(x_0,f(x_0)))$. To express this, we sometimes write $\hat{x}_t(x_0)$. In terms of the map $f$ and the measure $\pi_0$, the estimation criteria described above can be reformulated as follows.%

Eventual almost sure estimation:%
\begin{equation}\label{eq_E1}
  \pi_0( \{ x_0 \in X : d(f^t(x_0),\hat{x}_t(x_0)) \leq \ep \} ) = 1 \mbox{\quad for all\ } t \geq T.%
\end{equation}
 
Asymptotic almost sure estimation:%
\begin{equation}\label{eq_E2}
  \pi_0(\{x_0\in X : \limsup_{t\rightarrow\infty} d(f^t(x_0),\hat{x}_t(x_0)) \leq \ep \}) = 1.%
\end{equation}

Asymptotic estimation in expectation:%
\begin{equation}\label{eq_E3}
  \limsup_{t\rightarrow\infty}\int_X \pi_0(\rmd x_0) d(f^t(x_0),\hat{x}_t(x_0))^p \leq \ep.%
\end{equation}

It is easy to see that \eqref{eq_E1} implies both \eqref{eq_E2} and \eqref{eq_E3}, and in turn, \eqref{eq_E2} implies \eqref{eq_E3} due to the compactness of $X$. We will see that this order is monotonically reflected in the conditions required on the information transmission rates for the criteria to be satisfied.%

\subsection{Literature Review}

The results developed and the setup considered in this paper are related to the following three general areas.

{\bf Relations between dynamical systems, ergodic theory, and information theory.}
There has been a mutually beneficial relation between the ergodic theory of dynamical systems and information theory (see, e.g., \cite{ShieldsIT,Downarowicz} for comprehensive reviews). Information-theoretic tools have been very effective in answering many questions on the behavior of dynamical systems, for example the metric (also known as Kolmogorov-Sinai or measure-theoretic) entropy is crucial in the celebrated Shannon-McMillan-Breiman theorem as well as two important representation theorems: Ornstein's (isomorphism) theorem and the Krieger's generator theorem \cite{GrayProbabilit,Orn,Ornstein,katok2007fifty,Downarowicz}. The concept of sliding block encoding \cite{GrayIT} can be viewed as a stationary encoding of a dynamical system defined by the shift process, leading to fundamental results on the existence of stationary codes which perform as good as the limit performance of a sequence of optimal block codes. For topological dynamical systems, the theory of entropy structures and symbolic extensions answers the question to which extent a system can be represented by a symbolic system (under preservation of some topological structure), cf.~\cite{Downarowicz} for an overview of this theory. Entropy concepts, as is well-known in the information theory community, have extensive operational practical usage in identifying fundamental limits on source and channel coding for a very large class of sources \cite{ShieldsIT}. We refer the reader to \cite{GrayIT} and \cite{GrayNeuhoff} for further relations and discussions which also include information-theoretic results for processes which are not necessarily stationary.%

In this paper, we will provide further connections between the ergodic theory of dynamical systems and information theory by answering the problems posed in the previous section and relating the answers to the concepts of either metric or topological entropy. As we note later in the paper, some of the results presented here can be seen in analogy to certain fundamental results of ergodic theory about representing and approximating a dynamical system by a symbolic system, i.e., a subshift of a full shift on $n$ symbols.%

{\bf Zero-delay coding.}
In our setup, we will impose causality as a restriction in coding and decoding. Structural results for a finite horizon coding problem of a Markov source have been developed in a number of important papers, including the classic works by Witsenhausen \cite{Witsenhausen} and Walrand and Varaiya \cite{WalrandVaraiya} with extensions by Teneketzis \cite{Teneketzis}. The findings of \cite{WalrandVaraiya} have been generalized to continuous sources in \cite{YukIT2010arXiv} (see also \cite{YukLinZeroDelay} and \cite{BorkarMitterTatikonda}, where the latter imposes a structure apriori); and the structural results on optimal fixed-rate coding in \cite{Witsenhausen} and \cite{WalrandVaraiya} have been shown to be applicable to setups when one also allows for variable-length source coding in \cite{KMe1}. Related work also includes \cite{wood2016optimal,MahTen09,AsnaniWeissman,javidi2013dynamic} which have primarily considered the coding of discrete sources. \cite{wood2016optimal,YukLinZeroDelay,AsnaniWeissman,BorkarMitterTatikonda} have considered infinite horizon problems and in particular \cite{wood2016optimal} has established the optimality of stationary and deterministic policies for finite aperiodic and irreducible Markov sources. A related lossy coding procedure was introduced by Neuhoff and Gilbert \cite{NeuhoffGilbert}, called {\it causal source coding}, which has a different operational definition since delays in coding and decoding are allowed so that efficiencies through entropy coding can be utilized (though the causality condition is still preserved). We refer the reader to \cite{KMe2, KMe3} for further setups on zero-delay or causal coding in the multi-user as well as secure communications contexts. Further discussions on the literature are available in \cite{KMe1,YukLinZeroDelay,AsnaniWeissman}.

Among those that are most relevant to our paper is \cite{LinderZamir} where causal coding under a high rate assumption for stationary sources and individual sequences was studied. Linder and Zamir, in \cite{LinderZamir}, among many other results, established asymptotic quantitative relations between the differential entropy rate and the entropy rate of a uniformly and memorylessly quantized stationary process, and through this analysis established the near optimality of uniform quantizers in the low distortion regime. This result does not apply to the system we consider, since this system has a differential entropy rate of $-\infty$.%

{\bf Networked control.} In networked control, there has been an interest in identifying limitations on state estimation under information constraints. The results in this area have typically involved linear systems, and in the non-linear case the studies have only been on deterministic systems estimated/controlled over deterministic channels, with few exceptions. State estimation over discrete noiseless channels was studied in \cite{Matveev,MatveevSavkin} for linear discrete-time systems in a stochastic framework with the objective to bound the estimation error in probability. In these works, the inequality%
\begin{equation}\label{eq_dataratethm}
  C \geq H(A) := \sum_{\lambda \in \sigma(A)}\max\{0,n_{\lambda}\log|\lambda|\}%
\end{equation}
for the channel capacity $C$ was obtained as a necessary and almost sufficient condition. Here $A$ is the dynamical matrix of the system and the summation is over its eigenvalues $\lambda$ with multiplicities $n_{\lambda}$. This result is also in agreement with related data-rate theorem bounds established earlier in the literature \cite{Brockett,TatikondaThesis,NairEvans}. Some relevant studies that have considered non-linear systems are the following. The papers \cite{liberzon2016entropy}, \cite{pogromsky2011topological} and \cite{MPo} studied state estimation for non-linear deterministic systems and noise-free channels. In \cite{liberzon2016entropy}, Liberzon and Mitra characterized the critical data rate $C_0$ for exponential state estimation with a given exponent $\alpha\geq0$ for a continuous-time system on a compact subset $K$ of its state space. As a measure for $C_0$, they introduced a quantity called estimation entropy $h_{\mathrm{est}}(\alpha,K)$, which equals the topological entropy on $K$ in case $\alpha=0$, but for $\alpha>0$ is no longer a purely topological quantity. The paper \cite{kawan2016state} provided a lower bound on $h_{\mathrm{est}}(\alpha,K)$ in terms of Lyapunov exponents under the assumption that the system preserves a smooth measure. In \cite{MPo}, Matveev and Pogromsky studied three estimation objectives of increasing strength for discrete-time non-linear systems. For the weakest one, the smallest bit rate was shown to be equal to the topological entropy. For the other ones, general upper and lower bounds were obtained which can be computed directly in terms of the linearized right-hand side of the equation generating the system. A further related paper is \cite{Savkin06}, which studies state estimation for a class of non-linear systems over noise-free digital channels. Here also connections with topological entropy are established.%

In recent work \cite{kawanYukselisit17} and \cite{KYu}, we considered the state estimation problem for non-linear systems driven by noise and established negative bounds due to the presence of noise by viewing the noisy system as an infinite-dimensional random dynamical system subjected to the shift operation. In this paper, we relax the presence of noise, and are able to obtain positive results involving both the metric and the topological entropy.%

A related problem is the control of non-linear systems over communication channels. This problem has been studied in few publications, and mainly for deterministic systems and/or deterministic channels. Recently, \cite{yuksel2015stability} studied stochastic stability properties for a more general class of stochastic non-linear systems building on information-theoretic bounds and Markov-chain-theoretic constructions, however these bounds do not distinguish between the unstable and stable components of the tangent space associated with a dynamical non-linear system \cite[Sec.~4]{KYu}, except for the linear system case. Our paper here provides such a refinement, but only for estimation problems and in the low-distortion regime. 

{\bf Contributions.} In view of this literature review, we make the following contributions. We establish that for (E1), the topological entropy on the support of the measure $\pi_0$ essentially provides upper and lower bounds, and for (E3) the metric entropy essentially provides the relevant figure of merit for both upper and lower bounds. For (E2), the lower bound is provided by the metric entropy, whereas the upper bound can be given by either the topological entropy or by the metric entropy depending on the properties of the dynamical system. Through the analysis, we provide further connections between information theory and dynamical systems by identifying the operational usage of entropy concepts for three different estimation criteria. We obtain explicit bounds on the performance of optimal zero-delay codes; and regarding the contributions in networked control, we provide refined upper and lower bounds when compared with the existing literature.%

{\bf Organization of the paper.} In Section \ref{sec_preliminaries}, we introduce notation and recall relevant entropy concepts and related results from ergodic theory. Sections \ref{sec_e1}, \ref{sec_e2} and \ref{sec_e3} contain our main results on lower and upper estimates of the channel capacity for the estimation objectives (E1)--(E3). Section \ref{sec_discussion} provides a discussion of our results and relates them to similar results in the literature. In Section \ref{sec_corex}, we formulate some corollaries for systems in the differentiable category, relating the critical channel capacities to Lyapunov exponents, and we also present several examples. Finally, some fundamental theorems in ergodic theory used in our proofs are collected in the appendix, Section \ref{sec_appendix}.%

\section{Notation and review of relevant results from dynamical systems}\label{sec_preliminaries}%

Throughout the paper, all logarithms are taken to the base $2$. We define $\log \infty := \infty$. By $|A|$ we denote the cardinality of a set $A$. If $(X,d)$ is a metric space and $\emptyset \neq A \subset X$, we write $\diam A := \sup\{d(x,y) : x,y\in A\}$ for the diameter of $A$. If $(x_t)_{t\in\Z_+}$ is a sequence, we write $x_{[0,t]} = (x_0,x_1,\ldots,x_t)$ for any $t\geq0$.%

Let $f:X \rightarrow X$ be a uniformly continuous map on a metric space $(X,d)$. We define the iterates of $f$ recursively by%
\begin{equation*}
  f^0 := \mathrm{id}_X,\quad f^{n+1} := f \circ f^n,\quad \forall n \in \N.%
\end{equation*}
If $f$ is invertible, we additionally put $f^{-n} := (f^{-1})^n$ for all $n\in\N$.%

For $n\in\N$ and $\ep>0$, a subset $F \subset X$ is said to $(n,\ep)$-span another subset $K \subset X$ provided that for each $x\in K$ there is $y\in F$ with $d(f^i(x),f^i(y)) \leq \ep$ for $0 \leq i < n$. Alternatively, we can describe this in terms of the metric%
\begin{equation*}
  d_{n,f}(x,y) := \max_{0\leq i < n}d(f^i(x),f^i(y)),%
\end{equation*}
which is topologically equivalent to $d$. Any $d_{n,f}$-ball of radius $\ep>0$ is called a \emph{Bowen-ball of order $n$ and radius $\ep$}. In these terms, a set $F$ $(n,\ep)$-spans $K$ if the closed Bowen-balls of order $n$ and radius $\ep$ centered at the points in $F$ form a cover of $K$. If $K$ is compact, we write $r_{\spn}(n,\ep;K)$ to denote the minimal cardinality of a set which $(n,\ep)$-spans $K$, observing that this number is finite. The \emph{topological entropy of $f$ on $K$} is then defined by%
\begin{equation*}
  h_{\tp}(f;K) := \lim_{\ep\downarrow0}\limsup_{n\rightarrow\infty}\frac{1}{n}\log r_{\spn}(n,\ep;K).%
\end{equation*}
By the monotonicity properties of $r_{\spn}(n,\ep;K)$, the limit for $\ep\downarrow0$ can be replaced by the supremum over $\ep>0$. An alternative definition can be given in terms of $(n,\ep)$-separated sets. A subset $E \subset X$ is called \emph{$(n,\ep)$-separated} if for any $x,y\in E$ with $x\neq y$ one has $d_{n,f}(x,y) > \ep$. Writing $r_{\sep}(n,\ep;K)$ for the maximal cardinality of an $(n,\ep)$-separated subset of a compact set $K$, one finds that $r_{\spn}(n,\ep;K) \leq r_{\sep}(n,\ep;K) \leq r_{\spn}(n,\ep/2;K)$, and hence%
\begin{equation*}
  h_{\tp}(f;K) = \lim_{\ep\downarrow0}\limsup_{n\rightarrow\infty}\frac{1}{n}\log r_{\sep}(n,\ep;K).%
\end{equation*}
If $X$ is compact, we also write $h_{\tp}(f) = h_{\tp}(f;X)$ and call this number the \emph{topological entropy of $f$}.%

Now assume that $X$ is compact and $\mu$ is an $f$-invariant Borel probability measure on $X$, i.e., $\mu \circ f^{-1} = \mu$. Let $\PC$ be a finite Borel partition of $X$ and put%
\begin{equation}\label{eq_partition_entropy}
  h_{\mu}(f;\PC) := \lim_{n\rightarrow\infty}\frac{1}{n}H_{\mu}(\PC^n),\quad \PC^n := \bigvee_{i=0}^{n-1}f^{-i}\PC,%
\end{equation}
where $H_{\mu}(\PC^n) = -\sum_{P\in\PC^n}\mu(P)\log \mu(P)$ denotes the entropy of the partition $\PC^n$ and $\vee$ is the usual join operation for partitions, i.e., $\bigvee_{i=0}^{n-1}f^{-i}\PC$ is the partition of $X$ whose members are the sets%
\begin{equation*}
  P_{i_0} \cap f^{-1}(P_{i_1}) \cap \ldots \cap f^{-n+1}(P_{i_{n-1}}),\quad P_{i_j} \in \PC.%
\end{equation*}

The existence of the limit in \eqref{eq_partition_entropy} follows from subadditivity. The metric entropy of $f$ w.r.t.~$\mu$ is defined by%
\begin{equation*}
  h_{\mu}(f) := \sup_{\PC}h_{\mu}(f;\PC),%
\end{equation*}
the supremum taken over all finite Borel partitions $\PC$. If $\mu$ is ergodic (w.r.t.~$f$), i.e.,%
\begin{equation*}
  f^{-1}(A) = A \quad\Rightarrow\quad \mu(A) \in \{0,1\},%
\end{equation*}
there is an alternative characterization of $h_{\mu}(f)$ in terms of $(n,\ep)$-spanning sets, due to Katok \cite{Kat}. For a fixed $\delta \in (0,1)$ put%
\begin{align*}
   r_{\spn}(n,\ep,\delta) := 
 \min \bigl\{ |F| : & F\ (n,\ep)\mbox{-spans a set of}\\
&\qquad \mbox{$\mu$-measure} \geq 1-\delta \bigr\}.%
\end{align*}
Then the metric entropy of $f$ satisfies%
\begin{equation*}
  h_{\mu}(f) = \lim_{\ep\downarrow0}\limsup_{n\rightarrow\infty}\frac{1}{n}\log r_{\spn}(n,\ep,\delta),%
\end{equation*}
and this limit is independent of $\delta$, see \cite[Thm.~1.1]{Kat}.%

Note that both $h_{\mu}(f)$ are $h_{\tp}(f)$ are nonnegative quantities, which can assume the value $+\infty$. They are related by the variational principle, i.e.,%
\begin{equation*}
  h_{\tp}(f) = \sup_{\mu}h_{\mu}(f),%
\end{equation*}
where the supremum is taken over all $f$-invariant Borel probability measures (see \cite{Mis} for a short proof).%

In the rest of the paper, we will assume throughout that the metric space $(X,d)$ is compact.%

\section{Eventual almost sure estimation}\label{sec_e1}%

For the estimation objective (E1) we can characterize $C_0$ as follows.%

\begin{theorem}\label{thm_e1_bounds}
Assume that $\pi_0$ is $f$-invariant and $h_{\tp}(f;\supp\pi_0) < \infty$. Then the smallest channel capacity above which (E1) can be achieved for every $\ep>0$ satisfies%
\begin{equation*}
  h_{\tp}(f;\supp\pi_0) \leq C_0 \leq \log(1+\lfloor 2^{h_{\tp}(f;\supp\pi_0)} \rfloor).%
\end{equation*}
\end{theorem}


\begin{proof}
To prove the lower bound, assume that for some $\ep>0$ the objective is achieved under some coding and estimation policies over a channel of capacity $C = \log|\MC|$. That is, there exists $T = T(\ep)$ such that the set%
\begin{equation*}
  \tilde{X} := \{ x_0 \in X : d(f^t(x_0),\hat{x}_t(x_0)) \leq \ep,\ \forall t \geq T \}%
\end{equation*}
has measure one. For a fixed $n\in\N$, let $E_n \subset \supp\pi_0$ be a maximal set with the property that for all $x,y \in E_n$ with $x \neq y$ we have $d(f^t(x),f^t(y)) > 2\ep$ for some $t \in \{T,T+1,\ldots,T+n-1\}$. Since $\tilde{X}$ is dense in $\supp\pi_0$, a slight perturbation of $E_n$ leads to a subset of $\tilde{X}$ with the same property and the same cardinality. Hence, we may assume that $E_n \subset \tilde{X}$. Furthermore, consider the set%
\begin{equation*}
  \EC_{T+n} := \left\{ (\hat{x}_0,\hat{x}_1,\ldots,\hat{x}_{T+n-1}) : x_0 \in X \right\}%
\end{equation*}
of all possible $(T+n)$-strings the estimator can generate in the first time interval of length $T+n$. Define a map $\zeta:E_n \rightarrow \EC_{T+n}$ by assigning to each $x_0 \in E$ the corresponding $(T+n)$-string of estimates that is generated when the system starts in $x_0$. Assuming $\zeta(x) = \zeta(y)$, we find that for all $T \leq t \leq T + n - 1$,%
\begin{equation*}
  d(f^t(x),f^t(y)) \leq d(f^t(x),\hat{x}_t) + d(\hat{x}_t,f^t(y)) \leq \ep + \ep = 2\ep,%
\end{equation*}
implying that $x = y$. Hence, $\zeta$ is injective and thus, $|E_n| \leq |\EC_{T+n}| \leq |\MC|^{T+n}$. Now observe that the set $f^T(E) \subset \supp\pi_0$ is a maximal $(n,2\ep)$-separated set. This implies%
\begin{equation*}
  C = \log|\MC| \geq \frac{1}{T+n} \log r_{\sep}(n,2\ep;\supp\pi_0).%
\end{equation*}
Since $n$ and $\ep$ were chosen arbitrarily, letting $n$ tend to infinity and $\ep$ to zero, the inequality $C \geq h_{\tp}(f;\supp\pi_0)$ follows.%

The upper bound follows by applying the coding and estimation scheme described in \cite[Thm.~8]{MPo} for the system $(\supp\pi_0,f_{|\supp\pi_0})$, using that $\supp\pi_0$ is a compact $f$-invariant set. This scheme works as follows. Assume that the channel alphabet $\MC$ satisfies%
\begin{equation*}
  |\MC| \geq 1 + \lfloor 2^{h_{\tp}(f;\supp\pi_0)} \rfloor,%
\end{equation*}
implying $C > h_{\tp}(f;\supp\pi_0)$. Let $\ep>0$ and fix $k\in\N$ such that%
\begin{equation}\label{eq_kchoice}
  \frac{1}{k}\log r_{\spn}(k,\ep;\supp\pi_0) > C,%
\end{equation}
which is possible by the definition of $h_{\tp}$. Let $S$ be a set which $(k,\ep)$-spans $\supp\pi_0$ with $|S| = r_{\spn}(k,\ep;\supp\pi_0)$. The coder at time $t_j := jk$ computes $f^k(x_{jk})$ and chooses an element $y\in S$ such that $d(f^t(f^k(x_{jk})),f^t(y)) \leq \ep$ for $0 \leq t < k$ provided that $f^k(x_{jk}) \in \supp\pi_0$. By \eqref{eq_kchoice}, $y$ can be encoded and sent over the channel in the forthcoming time interval of length $k$. In the case when $f^k(x_{jk}) \notin \supp\pi_0$, it is not important what the coder sends. In the first case, at time $t_{j+1}$, the estimator has $y$ available and can use $f^t(y)$ as an estimate for $x_{(j+1)k+t}$. In this way, the estimation objective is achieved for all $x_0 \in \supp\pi_0$ and $t \geq k$. Hence,%
\begin{align*}
  P(d(x_t,\hat{x}_t) \leq \ep) &= \pi_0\left(\left\{ x_0 \in X : d(f^t(x_0),\hat{x}_t(x_0)) \leq \ep \right\}\right)\\
	&= \pi_0(\supp\pi_0) = 1%
\end{align*}
for all $t\geq k$, which completes the proof.%
\end{proof}

The value $+\infty$ for $h_{\tp}(f;\supp\pi_0)$, which is excluded in the above theorem, refers to the case, when the estimation objective cannot be achieved via a finite-capacity channel, which becomes clear from the first part of the proof.%

\section{Asymptotic almost sure estimation}\label{sec_e2}%

In this section, we study the asymptotic estimation objective (E2). As it turns out, the associated $C_0$ is always bounded below by the metric entropy $h_{\pi_0}(f)$ when $\pi_0$ is $f$-invariant. Since (E1) implies (E2), the topological entropy on $\supp\pi_0$ still provides an upper bound on $C_0$. Under an additional condition for the measure-theoretic dynamical system $(X,f,\pi_0)$, we are able to improve the upper bound by putting $h_{\pi_0}(f)$ in the place of $h_{\tp}(f;\supp\pi_0)$.%

\begin{theorem}\label{thm_e2_ub1}
Assume that $\pi_0$ is $f$-invariant. Then the smallest channel capacity above which (E2) can be achieved for every $\ep>0$ satisfies%
\begin{equation*}
  C_0 \leq \log(1 + \lfloor 2^{h_{\tp}(f;\supp\pi_0)} \rfloor).%
\end{equation*}
\end{theorem}

\begin{proof}
This follows immediately from Theorem \ref{thm_e1_bounds}, since (E1) is stronger than (E2).%
\end{proof}

\begin{theorem}\label{thm_e2_ub2}
Assume that $\pi_0$ is an ergodic measure for $f$. Furthermore, assume that there exists a finite Borel partition $\PC$ of $X$ of arbitrarily small diameter such that $h_{\pi_0}(f) = h := h_{\pi_0}(f;\PC)$ and%
\begin{equation}\label{eq_entropy_speed}
  \limsup_{n\rightarrow\infty}\frac{1}{n}\log\pi_0\bigl(\bigl\{ x : \bigl|-\frac{1}{n}\log\pi_0(\PC^n(x)) - h\bigr| > \delta \bigr\}\bigr) < 0%
\end{equation}
for all sufficiently small $\delta>0$, where $\PC^n(x)$ denotes the unique element of $\PC^n$ containing $x$. Then the smallest channel capacity $C_0$ above which (E2) can be achieved for every $\ep>0$ satisfies%
\begin{equation*}
  C_0 \leq \log(1 + \lfloor 2^{h_{\pi_0}(f)} \rfloor).%
\end{equation*}
\end{theorem}

\begin{proof}
The proof is subdivided into two steps.%

\emph{Step 1.} Without loss of generality, we may assume that $h_{\pi_0}(f) < \infty$, since otherwise the statement trivially holds. Consider a channel with input alphabet $\MC$ satisfying%
\begin{equation*}
  |\MC| \geq 1 + \lfloor 2^{h_{\pi_0}(f)} \rfloor,%
\end{equation*}
which implies $C > h_{\pi_0}(f)$, and fix $\ep>0$. Choose a number $\delta>0$ satisfying%
\begin{equation}\label{eq_delta_choice_e2}
  C \geq h_{\pi_0}(f) + \delta.%
\end{equation}
Using uniform continuity of $f$ on the compact space $X$, we find $\rho \in (0,\ep)$ so that%
\begin{equation}\label{eq_rho_choice_e2}
  d(x,y) < \rho \quad\Rightarrow\quad d(f(x),f(y)) < \ep.%
\end{equation}
Let $\PC$ be a finite Borel partition of $X$ satisfying the assumption \eqref{eq_entropy_speed} whose elements have diameter smaller than $\rho$ and put%
\begin{equation*}
  h := h_{\pi_0}(f;\PC) = \lim_{n\rightarrow\infty}\frac{1}{n}H_{\pi_0}(\PC^n).%
\end{equation*}
Recall that by ergodicity of $\pi_0$, the Shannon-McMillan-Breiman Theorem states that%
\begin{equation*}
  -\frac{1}{n}\log\pi_0(\PC^n(x)) \rightarrow h \mbox{\quad for a.e.\ } x \in X,%
\end{equation*}
implying that%
\begin{equation*}
  \pi_0\bigl( \bigl\{ x : \bigl|-\frac{1}{n}\log\pi_0(\PC^n(x)) - h\bigr| > \delta \bigr\}\bigr) \rightarrow 0%
\end{equation*}
as $n\rightarrow\infty$. The assumption \eqref{eq_entropy_speed} tells us that this convergence is exponential.%

Let $\PC = \{P_1,\ldots,P_r\}$ and write $i^n(x)$ for the length-$n$ itinerary of $x$ w.r.t.~$\PC$, i.e., $i^n(x) \in \{1,\ldots,r\}^n$ with $f^j(x) \in P_{i^n(x)_j}$ for $0 \leq j < n$. Consider the sets%
\begin{align*}
 \Sigma_{n,\delta} = \{ a \in \{1,\ldots,r\}^n : 2^{-n(h+\delta)} &\leq \pi_0(\{x:i^n(x)=a\})\\
& \leq 2^{-n(h-\delta)} \}.%
\end{align*}
We define sampling times by%
\begin{equation*}
  \tau_0 := 0,\quad \tau_{j+1} := \tau_j + j + 1,\quad j \geq 0%
\end{equation*}
and put%
\begin{equation*}
  \zeta_j := \pi_0\left(\left\{x\in X : i^j(f^{\tau_j}(x)) \notin \Sigma_{j,\delta}\right\}\right)%
\end{equation*}
for all $j\in\Z_+$. We obtain%
\begin{align*}
  1-\zeta_j &= \pi_0\bigl(\bigl\{x\in X : i^j(f^{\tau_j}(x)) \in \Sigma_{j,\delta}\bigr\}\bigr)\\
	          &= \pi_0\bigl(\bigl\{x\in X : 2^{-j(h+\delta)} \leq 
						   \pi_0(\{y:i^j(y)=i^j(f^{\tau_j}(x))\})\\
					 &\qquad \qquad \qquad \leq 2^{-j(h-\delta)} \bigr\}\bigr)\\
						&= \pi_0\bigl(\bigl\{x\in X : \bigl|-\frac{1}{j}\log \pi_0(\PC^j(f^{\tau_j}(x))) - h\bigr| \leq \delta \bigr\}\bigr)\\
						&= \pi_0\bigl(\bigl\{ x\in X : \bigl|-\frac{1}{j}\log \pi_0(\PC^j(x)) - h\bigr| \leq \delta \bigr\}\bigr),%
\end{align*}
where we used $f$-invariance of $\pi_0$ in the last equality. Then \eqref{eq_entropy_speed} implies the existence of a constant $\alpha>0$ so that for all sufficiently large $j$ we have%
\begin{align*}
  \zeta_j = \pi_0\bigl(\bigl\{x\in X : \bigl|-\frac{1}{j}\log \pi_0(\PC^j(x)) - h\bigr| > \delta \bigr\}\bigr) \leq \rme^{-j\alpha}.%
\end{align*}
Hence, for some $j_0 \geq 0$,%
\begin{equation*}
  \sum_{j=0}^{\infty}\zeta_j \leq \sum_{j=0}^{j_0}\zeta_j + \sum_{j=j_0+1}^{\infty}\rme^{-j\alpha} < \infty.%
\end{equation*}
Consequently, the Borel-Cantelli lemma implies%
\begin{equation}\label{eq_bc_e2}
  \pi_0(\{x \in X : \exists j_0 \mbox{ with } i^j(f^{\tau_j}(x))) \in \Sigma_{j,\delta},\ \forall j \geq j_0\}) = 1.%
\end{equation}

\emph{Step 3.} Using Corollary \ref{cor_smb} in the appendix, $h = h_{\pi_0}(f)$ and \eqref{eq_delta_choice_e2}, we obtain%
\begin{equation}\label{eq_typicalset_card_e2}
  |\Sigma_{j,\delta}| \leq 2^{j(h + \delta)} \leq 2^{jC} = |\MC|^j.%
\end{equation}
Hence, $j$ channel uses are sufficient to transmit an encoded element of $\Sigma_{j,\delta}$. Now we specify the coding and estimation policies. In the time interval from $\tau_j$ to $\tau_{j+1}-1$, which has length $\tau_{j+1} - \tau_j = j+1$, the coder encodes the information regarding the orbit in the time interval from $\tau_{j+1}$ to $\tau_{j+2}-2$, i.e., the itinerary%
\begin{equation*}
  i^{\tau_{j+2} - \tau_{j+1} - 1}(f^{\tau_{j+1}}(x)) = i^{j+1}(f^{\tau_{j+1}}(x)).%
\end{equation*}
For all $x$ with $i^{j+1}(f^{\tau_{j+1}}(x)) \in \Sigma_{j+1,\delta}$ this information can be sent through the channel in the time interval of length $j+1$ by \eqref{eq_typicalset_card_e2}. At time $\tau_{j+1} + k$, $0 \leq k \leq j$, the estimator output $\hat{x}_{\tau_{j+1}+k}$ is an arbitrary element of the partition set $P_{s_k}$ with $s_k$ being the symbol at position $k$ in the transmitted string $i^{j+1}(f^{\tau_{j+1}}(x))$. At time $\tau_{j+1} + j+1 = \tau_{j+2} - 1$, the estimator output is $\hat{x}_{\tau_{j+2}-1} := f(\hat{x}_{\tau_{j+2}-2})$. Provided that $i^{j+1}(f^{\tau_{j+1}}(x)) \in \Sigma_{j+1,\delta}$, the estimation accuracy of $\ep$ will be achieved, because%
\begin{equation*}
  d(\hat{x}_{\tau_{j+1}+k},f^{\tau_{j+1}+k}(x)) \leq \diam P_{s_k} < \rho < \ep,\quad \forall 0 \leq k \leq j%
\end{equation*}
by the choice of the partition $\PC$ and%
\begin{equation*}
  d(\hat{x}_{\tau_{j+2}-1},f^{\tau_{j+2}-1}(x)) = d(f(\hat{x}_{\tau_{j+2}-2}),f(f^{\tau_{j+2}-2}(x))) < \ep%
\end{equation*}
by \eqref{eq_rho_choice_e2}. According to \eqref{eq_bc_e2}, for almost every $x$ we have $i^j(f^{\tau_j}(x)) \in \Sigma_{j,\delta}$ for all sufficiently large $j$ so that the estimation accuracy $\ep$ will eventually be achieved for those $x$, implying%
\begin{equation*}
  \pi_0(\{x\in X : \limsup_{t\rightarrow\infty} d(f^t(x),\hat{x}_t(x)) \leq \ep \}) = 1,%
\end{equation*}
which completes the proof.%
\end{proof}

\begin{remark}
The exponential convergence \eqref{eq_entropy_speed} is certainly a strong assumption and we do not expect that it is satisfied by a great variety of dynamical systems. The standard example, where it can be verified, is a shift over a finite alphabet equipped with a Gibbs measure, cf.~\cite{VZh}. Since such shifts are often embedded in chaotic attractors of smooth systems, we also find examples in this class of systems. A concrete example will be given in Section \ref{sec_corex}. We also note that the (geometric/exponential convergence) condition (\ref{eq_entropy_speed}) can be relaxed; what is essential is the summability of the probability sequence $\sum_j \pi_0\bigl(\bigl\{x\in X : \bigl|-\frac{1}{j}\log \pi_0(\PC^j(x)) - h\bigr| > \delta \bigr\}\bigr)$ for every $\delta > 0$ so that the Borel-Cantelli lemma can be invoked. In particular, instead of a geometric convergence rate, a sufficiently fast subgeometric convergence would be sufficient.%
\end{remark}

To provide a lower bound on $C_0$, we will use the following lemma which can be found in \cite[Prop.~1]{New,Ne2}. For completeness, we provide its proof in the appendix.%

\begin{lemma}\label{lem_newhouse}
Assume that $\pi_0$ is $f$-invariant. Given $\tilde{X} \subset X$ and $\delta>0$, let%
\begin{equation*}
  h(f;\tilde{X},\delta) := \limsup_{n\rightarrow\infty}\frac{1}{n}\log r_{\sep}(n,\delta;\tilde{X}).%
\end{equation*}
Then, for any partition $\PC$ of $X$ of the form $\PC = \{P_1,\ldots,P_s,X\backslash \bigcup P_i\}$ with compact sets $P_1,\ldots,P_s$ ($s\in\N$) and for any $\rho>0$, we can choose $\delta>0$, $\alpha \in (0,1)$ and $N = N(\rho) \in \N$ so that for any measurable set $\tilde{X} \subset X$ with $\pi_0(\tilde{X}) \geq 1 - \alpha$ we have%
\begin{equation*}
  \frac{1}{N}h_{\pi_0}(f^N;\PC) \leq h(f;\tilde{X},\delta) + \rho.%
\end{equation*}
\end{lemma}
                                              
\begin{theorem}\label{thm_e2_lb}
Assume that $\pi_0$ is $f$-invariant. Then the smallest channel capacity above which (E2) can be achieved for every $\ep>0$ satisfies%
\begin{equation}\label{eq_im_lb}
  C_0 \geq h_{\pi_0}(f).%
\end{equation}
If $\pi_0$ is not invariant, but there exists an invariant measure $\pi_*$ which is absolutely continuous w.r.t.~$\pi_0$, then%
\begin{equation}\label{eq_nim_lb}
  C_0 \geq h_{\pi_*}(f).%
\end{equation}
Here the cases $h_{\pi_0}(f) = \infty$ and $h_{\pi_*}(f) = \infty$, respectively, are included. They correspond to the situation, when the estimation objective cannot be achieved over a finite-capacity channel.%
\end{theorem}

\begin{proof}
Consider a channel of capacity $C$ over which the objective (E2) can be achieved for every $\ep>0$. Fix coding and estimation policies achieving the objective for a fixed $\ep>0$. Pick $\delta>\ep$ and define for $x\in X$,%
\begin{equation*}
  T(x,\delta) := \min\left\{ k \in \N : d(f^t(x),\hat{x}_t(x)) \leq \delta \mbox{ for all } t \geq k \right\}.%
\end{equation*}
For every $K\in\N$ put $B^K(\delta) := \{x \in X : T(x,\delta) \leq K\}$. Since $x \mapsto d(f^t(x),\hat{x}_t(x))$ is a measurable map, $B^K(\delta)$ is a measurable set. From \eqref{eq_E2} it follows that%
\begin{equation}\label{eq_limitrel}
  \lim_{K\rightarrow\infty}\pi_0(B^K(\delta)) = \pi_0\Bigl(\bigcup_{K\in\N}B^K(\delta)\Bigr) = 1.%
\end{equation}
Now, for a given $\alpha \in (0,1)$, pick $K\in\N$ such that%
\begin{equation*}
  \pi_0(B^K(\delta)) \geq 1-\alpha.%
\end{equation*}
We consider the set $\tilde{X} := f^K(B^K(\delta))$. Observe that $x \in \tilde{X}$, $x = f^K(y)$, implies%
\begin{equation*}
  d(f^t(x),\hat{x}_{t+K}(y)) \leq \delta \mbox{\quad for all\ } t \geq 0.%
\end{equation*}
Now consider a maximal $(n,2\delta)$-separated set $E_n \subset \tilde{X}$ for some $n\in\N$. Recall that $\hat{x}_t = \gamma_t(q_0,q_1,\ldots,q_t)$. For arbitrary integers $m \geq m_0$, write%
\begin{align*}
  \hat{E}_{m_0,m} := &\{ (\gamma_{m_0}(q_{[0,m_0]}),\gamma_{m_0+1}(q_{[0,m_0+1]}),\ldots,\gamma_m(q_{[0,m]}))\\
		&\qquad : q_{[0,m]} \in \MC^{m+1} \}.%
\end{align*}
Define a map $\zeta:E_n \rightarrow \hat{E}_{K,K+n}$ by assigning to each $x\in E_n$ the unique element of $\hat{E}_{K,K+n}$ that is generated by the coding and estimation policy when the system starts in the initial state $x_0 = f^{-K}(x) \in B^K(\delta)$. We claim that $\zeta$ is injective. Indeed, assume that $\zeta(x_1) = \zeta(x_2)$ for some $x_1,x_2 \in E_n$. Then, for $0 \leq j \leq n$,%
\begin{align*}
  d(f^j(x_1),f^j(x_2)) &\leq d(f^j(x_1),\hat{x}_{K + j}) + d(\hat{x}_{K + j},f^j(x_2))\\
	&\leq \delta + \delta = 2\delta,%
\end{align*}
implying $x_1 = x_2$, because $E_n$ is $(n,2\delta)$-separated. Hence,%
\begin{equation*}
  \log |E_n| \leq \log |\hat{E}_{K,K+n}| \leq \log |\MC|^{K+n} = (K+n)C.%
\end{equation*}
We thus obtain%
\begin{align*}
  C &\geq \limsup_{n\rightarrow\infty}\frac{1}{K + n}\log|E_n|\\
	&= \limsup_{n\rightarrow\infty}\frac{1}{n}\log|E_n| = h(f;\tilde{X},2\delta),%
\end{align*}
where $\pi_0(\tilde{X}) \geq 1 - \alpha$ and $\tilde{X}$ depends on the choice of $\delta > \ep$.

By Lemma \ref{lem_newhouse}, for a partition $\PC = \{P_1,\ldots,P_s,X\backslash \bigcup P_i\}$ and $\rho>0$, we can choose $\ep>0$, $\delta>\ep$, $\alpha>0$ and $N = N(\rho)$ such that the above construction yields the inequality%
\begin{equation*}
  \frac{1}{N}h_{\pi_0}(f^N;\PC) \leq h(f;\tilde{X},2\delta) + \rho \leq C + \rho.%
\end{equation*}
Taking the supremum over all partitions $\PC$ as in the lemma and using the power rule $h_{\pi_0}(f^N) = N h_{\pi_0}(f)$ yields%
\begin{equation*}
  h_{\pi_0}(f) \leq C + \rho.%
\end{equation*}
Since $\rho>0$ can be chosen arbitrarily, the inequality \eqref{eq_im_lb} follows.%

Now assume that $\pi_*$ is an $f$-invariant probability measure which is absolutely continuous w.r.t.~$\pi_0$. Then for every $\alpha>0$ there is $\beta>0$ such that $\pi_0(B) < \beta$ implies $\pi_*(B) < \alpha$. Looking at the complements of the sets $B^K(\delta)$, we see that the measure $\pi_0$ can be replaced by $\pi_*$ in \eqref{eq_limitrel}, which implies the inequality \eqref{eq_nim_lb}.%
\end{proof}

\section{Asymptotic estimation in expectation}\label{sec_e3}%

Finally, we consider the estimation objective (E3). As we will see, under mild assumptions on the admissible coding and estimation policies, $C_0$ is characterized by the metric entropy $h_{\pi_0}(f)$.%

\begin{theorem}\label{thm_e3_ub}
Assume that $\pi_0$ is an ergodic measure for $f$. Then the smallest channel capacity above which (E3) can be achieved for every $\ep>0$ satisfies%
\begin{equation*}
  C_0 \leq \log(1+\lfloor 2^{h_{\pi_0}(f)}\rfloor).%
\end{equation*}
\end{theorem}

\begin{proof}
Without loss of generality, we may assume that $h_{\pi_0}(f) < \infty$, since otherwise the statement trivially holds. Consider a channel with input alphabet $\MC$ satisfying%
\begin{equation*}
  |\MC| \geq 1 + \lfloor 2^{h_{\pi_0}(f)} \rfloor,%
\end{equation*}
which implies $C > h_{\pi_0}(f)$, and fix $\ep>0$. Choose a number $\delta>0$ satisfying%
\begin{equation}\label{eq_delta_choice}
  C \geq h_{\pi_0}(f) + \delta.%
\end{equation}
Let $\PC = \{P_1,\ldots,P_r\}$ be a finite Borel partition of $X$ whose elements have diameter smaller than $(\ep/2)^{1/p}$ and put%
\begin{equation*}
  h := h_{\pi_0}(f;\PC) = \lim_{n\rightarrow\infty}\frac{1}{n}H_{\pi_0}(\PC^n).%
\end{equation*}
Using the assumption that $\pi_0$ is ergodic, the Shannon-McMillan-Breiman Theorem states that%
\begin{equation}\label{eq_smb}
  -\frac{1}{n}\log\pi_0(\PC^n(x)) \rightarrow h \mbox{\quad for a.e.\ } x \in X,%
\end{equation}
where $\PC^n(x)$ is the element of $\PC^n$ containing $x$. Write $i^n(x)$ for the length $n$ itinerary of $x$ w.r.t.~$\PC$ (see the appendix), and consider the sets%
\begin{align*}
  \Sigma_{n,\delta} = \{ a \in \{1,\ldots,r\}^n : 2^{-n(h+\delta)} &\leq \pi_0(\{x:i^n(x)=a\})\\
	&\leq 2^{-n(h-\delta)} \}.%
\end{align*}
By Corollary \ref{cor_smb} in the appendix, we have%
\begin{equation}\label{eq_fullmeasset}
  \pi_0(\{x\in X : \exists m\in\N \mbox{ with } i^n(x) \in \Sigma_{n,\delta},\ \forall n \geq m\}) = 1%
\end{equation}
and $|\Sigma_{n,\delta}| \leq 2^{n(h+\delta)}$. By \eqref{eq_fullmeasset}, we can fix $k$ large enough so that%
\begin{equation}\label{eq_goodset_meas}
  \pi_0(\{x\in X : i^k(x) \in \Sigma_{k,\delta}\}) \geq 1 - \frac{\ep}{2(\diam X)^p}.%
\end{equation}
Using that $h \leq h_{\pi_0}(f)$, we also obtain%
\begin{equation}\label{eq_typicalset_card}
  |\Sigma_{k,\delta}| \leq 2^{k(h + \delta)} \leq 2^{kC} = |\MC|^k.%
\end{equation}
Thus, $k$ channel uses are sufficient to transmit an encoded element of $\Sigma_{k,\delta}$. Now we specify the coding and estimation policy. For any $j\in\Z_+$, in the time interval from $jk$ to $(j+1)k-1$, the coder encodes the information regarding the orbit in the time interval from $(j+1)k$ to $(j+2)k-1$, i.e., the itinerary $i^k(f^{(j+1)k}(x))$. For all $x$ with $i^k(f^{(j+1)k}(x)) \in \Sigma_{k,\delta}$ this information can be sent through the channel in a time interval of length $k$ by \eqref{eq_typicalset_card}. At time $(j+1)k + i$, $0 \leq i < k$, the estimator output $\hat{x}_{(j+1)k+i}$ is an arbitrary element of the partition set $P_{s_i}$ with $s_i$ being the symbol at position $i$ in the transmitted string $i^k(f^{(j+1)k}(x))$. Provided that $i^k(f^{(j+1)k}(x)) \in \Sigma_{k,\delta}$, the estimation error satisfies%
\begin{equation*}
  d(\hat{x}_{(j+1)k+i},f^{(j+1)k+i}(x)) \leq \diam P_{s_i} < \left(\frac{\ep}{2}\right)^{1/p}%
\end{equation*}
for $0 \leq i < k$, by the choice of the partition $\PC$. Putting $X_1 := \{ x : i^k(f^{(j+1)k}(x)) \in \Sigma_{k,\delta} \}$ and $X_2 := X \backslash X_1$, and using \eqref{eq_goodset_meas} together with invariance of $\pi_0$ under $f$, we obtain%
\begin{align*}
&  E[d(x_{(j+1)k+i},\hat{x}_{(j+1)k+i})^p] \\
  &= \int_{X_1} \pi_0(\rmd x) d(f^{(j+1)k+i}(x),\hat{x}_{(j+1)k+i}(x))^p\\
	&\qquad + \int_{X_2} \pi_0(\rmd x) d(f^{(j+1)k+i}(x),\hat{x}_{(j+1)k+i}(x))^p\\
  &\leq \pi_0(X_1) \frac{\ep}{2} + \pi_0(X_2) (\diam X)^p \leq \frac{\ep}{2} + \frac{\ep}{2} = \ep.%
\end{align*}
Hence, for all $t \geq k$, we have $E[d(x_t,\hat{x}_t)^p] \leq \ep$, implying that the estimation objective is achieved.%
\end{proof}

A stationary source can be expressed as a mixture of ergodic components; if these can be transmitted separately, the proof would also hold without the condition of ergodicity \cite[Thm.~11.3.1]{GrayIT}. If the process is not ergodic, but stationary, and if there only exist finitely many ergodic invariant measures in the support of $\pi_0$, then the proof above can be trivially modified such that the encoder sends a special message to inform the estimator which ergodic subsource is transmitted, and the coding scheme for the ergodic subsource could be applied.%

To obtain the analogous lower bound, we need to introduce some notation. Let $\mu$ be an ergodic $f$-invariant Borel probability measure on $X$. For $x\in X$, $n\in\N$, $\ep>0$ and $r\in(0,1)$ define the sets%
\begin{align*}
  B(x,n,\ep,r) := \Bigl\{y\in X : &\frac{1}{n}\bigl|\bigl\{0\leq k<n : d(f^k(x),f^k(y))\\
	&\leq \ep\bigr\}\bigr| > 1-r \Bigr\}.%
\end{align*}
A set $F\subset X$ is called $(n,\ep,r,\delta)$-spanning for $\delta\in(0,1)$ if%
\begin{equation*}
  \mu\Bigl(\bigcup_{x\in F}B(x,n,\ep,r)\Bigr) \geq 1-\delta.%
\end{equation*}
We write $r_{\spn}(n,\ep,r,\delta)$ for the minimal cardinality of such a set and define%
\begin{equation*}
  h_r(f,\mu,\ep,\delta) := \limsup_{n\rightarrow\infty}\frac{1}{n}\log r_{\spn}(n,\ep,r,\delta).%
\end{equation*}

According to \cite[Thm.~A]{Rea}, for every $\rho\in(0,1)$ we have%
\begin{equation*}
  \lim_{r\downarrow0}\lim_{\ep\downarrow0}h_r(f,\mu,\ep,\rho) = h_{\mu}(f).%
\end{equation*}
Since $h_r(f,\mu,\ep,\rho)$ is increasing as $r\downarrow0$ and as $\ep\downarrow0$, we can write this as%
\begin{align*}
  h_{\mu}(f) = \sup_{r>0}\sup_{\ep>0}h_r(f,\mu,\ep,\rho) = \sup_{(r,\ep)\in\R_{>0}^2}h_r(f,\mu,\ep,\rho).%
\end{align*}
In particular, this implies%
\begin{equation}\label{eq_limit}
  h_{\mu}(f) = \lim_{\ep\downarrow0} h_{\varphi_1(\ep)}(f,\mu,\varphi_2(\ep),\rho)%
\end{equation}
for any functions $\varphi_1,\varphi_2$ with $\lim_{\ep\downarrow0}\varphi_i(\ep) = 0$, $i=1,2$.%

In the following theorem, we will restrict ourselves to periodic coding and estimation schemes using finite memory. 

\begin{theorem}\label{thm_e3_lb}%
Assume that $\pi_0$ is an ergodic measure for $f$. Additionally, assume that there exists $\tau>0$ so that the coder map $\delta_t$ is of the form%
\begin{equation*}
  q_t = \delta_t(x_{[t-\tau+1,t]})%
\end{equation*}
and $\delta_{t + \tau} \equiv \delta_t$. Further assume that the estimator map is of the form%
\begin{equation*}
  \hat{x}_t = \gamma_t(q_{[t-\tau+1,t]})%
\end{equation*}
and also $\gamma_{t+\tau} \equiv \gamma_t$. Then the smallest channel capacity above which (E3) can be achieved for every $\ep>0$ satisfies%
\begin{equation*}
  C_0 \geq h_{\pi_0}(f).%
\end{equation*}
\end{theorem}

Before presenting the proof, we note the following.

\begin{remark}
For optimal zero-delay codes, the structural results available for finite horizon problems (see e.g. \cite{Witsenhausen,WalrandVaraiya,Teneketzis,YukIT2010arXiv}  and infinite horizon problems \cite{wood2016optimal,YukLinZeroDelay,AsnaniWeissman} establish that an optimal encoder only uses the most recent state and some information variable (available at both the encoder and the decoder, such as the past transmitted symbols or a conditional measure on the state given the information at the decoder). In particular, \cite{wood2016optimal} established the optimality of stationary and deterministic policies under such a structure. On the other hand, \cite[Thm.~4]{wood2016optimal} shows that for finitely valued Markov sources which are aperiodic, irreducible, and stationary, periodic coding policies (as that considered in Theorem \ref{thm_e3_lb}) indeed perform arbitrarily well for the zero-delay setup also (with a rate of convergence of the order $\frac{1}{T}$ with $T$ being the period); that this holds for the non-causal setup is well known in the information theory literature through asymptotic optimality of block codes \cite{GrayIT}: In particular, for stationary and memoryless sources, a rate of convergence for such a result was shown to be of the type $O\big(\frac{\log T}{T}\big)$ \cite{pilc1967coding}, \cite{zhang1997redundancy}. See also \cite{kostina2012fixed} for a detailed literature review and further finite blocklength performance bounds. Thus, the use of finite memory policies in Theorem \ref{thm_e3_lb} is a practical construction, and it is our goal to investigate the optimal encoding in the absence of such an apriori assumption in the future.%
\end{remark}

\begin{proof}
The proof is subdivided into three steps.%

\emph{Step 1.} Fix $\ep>0$ and consider a coding and estimation policy that achieves \eqref{eq_E3} for $\ep/2$. Then we find $T\in\N$ so that%
\begin{equation}\label{eq_eventual_exp}
  E[d(x_t,\hat{x}_t)^p] \leq \ep \mbox{\quad for all\ } t \geq T.%
\end{equation}
For each $t \geq T$ we decompose $X$ into the two sets%
\begin{align*}
  X_1^t &:= \left\{ z\in X : d(z,\hat{x}_t(f^{-t}(z)))^p < \sqrt{\ep} \right\},\\
	X_2^t &:= \left\{ z\in X : d(z,\hat{x}_t(f^{-t}(z)))^p \geq \sqrt{\ep} \right\}.%
\end{align*}
Now assume to the contrary that for some $t\geq T$,%
\begin{equation}\label{eq_contra_ass}
  \pi_0(X_2^t) > \sqrt{\ep}.%
\end{equation}
In contradiction to \eqref{eq_eventual_exp}, this implies%
\begin{align*}
  E[d(x_t,\hat{x}_t)^p] &= \int_X \pi_0(\rmd x) d(f^t(x),\hat{x}_t(x))^p\allowdisplaybreaks\\
	&= \int_X (f^{-t}_*\pi_0)(\rmd x) d(f^t(x),\hat{x}_t(x))^p\allowdisplaybreaks\\
	&= \int_X \pi_0(\rmd z) d(z,\hat{x}_t(f^{-t}(z)))^p\allowdisplaybreaks\\
	&\geq \int_{X_2^t} \pi_0(\rmd z) d(z,\hat{x}_t(f^{-t}(z)))^p\allowdisplaybreaks\\
	&\geq \pi_0(X_2^t) \sqrt{\ep} \stackrel{\eqref{eq_contra_ass}}{>} \sqrt{\ep}^2 = \ep,%
\end{align*}
where we used that $\pi_0$ is $f$-invariant. Hence, we obtain%
\begin{equation}\label{eq_x2_meas}
  \pi_0(X_2^t) \leq \sqrt{\ep}.%
\end{equation}

\emph{Step 2.} By Birkhoff's Ergodic Theorem, the measure of a set determines how frequently this set is visited over time. We want to know the frequency of how often the sets $X_2^t$ are visited by the trajectories of $f$. More precisely, we want to achieve that%
\begin{equation}\label{eq_birkhoff_na}
  \limsup_{n\rightarrow\infty}\frac{1}{n}\sum_{i=0}^{n-1}\chi_{X_2^{T+i}}(f^{T+i}(x)) \leq \sqrt{\ep}%
\end{equation}
for $\pi_0$-almost all $x\in X$, say for all $x \in \tilde{X} \subset X$ with $\pi_0(\tilde{X}) = 1$. We can prove \eqref{eq_birkhoff_na} under the assumption about the coding and estimation policies. Using the notation $\alpha_t(z) := d(z,\hat{x}_t(f^{-t}(z)))^2$, the set $X_2^t$ can be written as%
\begin{equation*}
  X_2^t := \alpha_t^{-1}([\sqrt{\ep},\infty)).%
\end{equation*}
By our assumptions,%
\begin{align*}
  \hat{x}_t &= \gamma_t(  [\delta_s(x_{[s-\tau+1,s]})]_{s\in [t-\tau+1,t]})\\
	&= \gamma_t( [\delta_s( f^{s-\tau+1}(x_0),\ldots,f^s(x_0) )]_{s\in [t-\tau+1,t]}).%
\end{align*}
Hence,%
\begin{align*}
&  \hat{x}_{t+\tau}(x_0) \\
  &= \gamma_{t+\tau}( [\delta_s( f^{s-\tau+1}(x_0),\ldots,f^s(x_0) )]_{s\in [t+1,t+\tau]})\\
	&= \gamma_t( [\delta_s(f^{s+1}(x_0),\ldots,f^{s+\tau}(x_0) )]_{s \in [t-\tau+1,t]})\\
	&= \gamma_t( [\delta_s(f^{s-\tau+1} (f^{\tau}(x_0) ),\ldots,f^s(f^{\tau}(x_0)))]_{s\in[t-\tau+1,t]}) \\
	&	= \hat{x}_t(f^{\tau}(x_0)),%
\end{align*}
implying $\hat{x}_{t+\tau} = \hat{x}_t \circ f^{\tau}$ for all $t \in \Z_+$, and thus%
\begin{align*}
&  \alpha_{t+\tau}(z) = d(z,\hat{x}_{t+\tau}(f^{-(t+\tau)}(z)))^2 = d(z,\hat{x}_t(f^{-t}(z)))^2 \\
& \quad \quad = \alpha_t(z).%
\end{align*}
Hence, we have the same periodicity for the sets $X_2^t$, i.e.,%
\begin{equation*}
  X_2^{t+\tau} \equiv X_2^t,\quad \forall t\in\Z_+.%
\end{equation*}
Writing each $n\in\N$ as $n = k_n \tau + r_n$ with $k_n \in \Z_+$ and $0 \leq r_n < \tau$, we find that%
\begin{align*}
  & \frac{1}{n}\sum_{i=0}^{n-1}\chi_{X_2^{T+i}}(f^{T+i}(x))  \\
  &\leq \frac{1}{\tau}\sum_{j=0}^{\tau-1} \frac{1}{k_n}  \sum_{i=0}^{k_n-1}  \chi_{X_2^{T+j}}(f^{T+i\tau+j}(x)) + \frac{\tau}{n}.%
\end{align*}
Hence, Birkhoff's Ergodic Theorem applied to $f^{\tau}$ implies%
\begin{equation}
  \lim_{n\rightarrow\infty}\frac{1}{n}\sum_{i=0}^{n-1}\chi_{X_2^{T+i}}(f^{T+i}(x)) \leq \frac{1}{\tau}\sum_{j=0}^{\tau-1} \pi_0(X_2^{T+j}) \leq \sqrt{\ep} \label{BET1}
\end{equation}
for all $x$ from a full measure set $\tilde{X}\subset X$. Now, we can write $\tilde{X}$ as the union of the sets%
\begin{equation*}
  \tilde{X}_m := \left\{ x\in X : \frac{1}{n}\sum_{i=0}^{n-1}\chi_{X_2^{T+i}}(f^{T+i}(x)) \leq 2\sqrt{\ep},\ \forall n \geq m \right\}.%
\end{equation*}
Since $\tilde{X}_m \subset \tilde{X}_{m+1}$ for all $m$, by \eqref{BET1} we find $m_0$ with $\pi_0(\tilde{X}_{m_0}) \geq 1 - \rho$, where $\rho$ is a given number in $(0,1)$. Hence, we have%
\begin{equation*}
  \frac{1}{n}\sum_{i=0}^{n-1}\chi_{X_2^{T+i}}(f^{T+i}(x)) \leq 2\sqrt{\ep} \mbox{\quad for all\ } n \geq m_0,\ x \in \tilde{X}_{m_0}.%
\end{equation*}
Then for $n \geq m_0$, the number of $i$'s in $\{0,1,\ldots,n-1\}$ such that $f^{T+i}(x) \in X_2^{T+i}$ is $\leq 2n\sqrt{\ep}$.%

\emph{Step 3.} For every $n\geq m_0$ consider the set%
\begin{align*}
  E_n := \bigl\{ (\hat{x}_T(x),\hat{x}_{T+1}(x),\ldots,\hat{x}_{T+n-1}(x)) \in X^n : x \in \tilde{X}_{m_0} \bigr\}.%
\end{align*}
The cardinality of $E_n$ is dominated by the largest possible number of coding symbols the estimator has available at time $T+n-1$, i.e.,%
\begin{equation*}
  |E_n| \leq |\MC|^{T+n}.%
\end{equation*}
Define for each $n\in\N$ a set $F_n \subset X$ with $|F_n| = |E_n|$ by choosing for each element%
\begin{equation*}
  (\hat{x}_T(x),\hat{x}_{T+1}(x),\ldots,\hat{x}_{T+n-1}(x)) \in E_n%
\end{equation*}
a unique representative $x \in \tilde{X}_{m_0}$ and letting $F_n$ be the set of all these representatives.%

Now pick an arbitrary $y\in \tilde{X}_{m_0}$ and consider $(\hat{x}_T(y),\ldots,\hat{x}_{T+n-1}(y)) \in E_n$. Let $x \in F_n$ be the unique representative of this string, implying $\hat{x}_{T+i}(y) = \hat{x}_{T+i}(x)$ for $0 \leq i < n$. Then%
\begin{align}\label{eq_triangle}
  d(f^{T+i}(y),f^{T+i}(x)) &\leq d(f^{T+i}(y),\hat{x}_{T+i}(y)) \nonumber\\
	&\ + d(\hat{x}_{T+i}(x),f^{T+i}(x))%
\end{align}
holds for $0 \leq i < n$. From Step 2 we know that $f^{T+i}(y) \in X_1^{T+i}$ for $\geq n(1 - 2\sqrt{\ep})$ $i$'s and the same holds for $x$ in place of $y$. Hence, the number if $i$'s, where both holds, is $\geq n(1 - 4 \sqrt{\ep})$. In this case, \eqref{eq_triangle} implies%
\begin{equation*}
  d(f^{T+i}(y),f^{T+i}(x)) \leq 2\sqrt[2p]{\ep}.%
\end{equation*}
Hence,%
\begin{equation*}
  f^T(y) \in B(f^T(x),n,2\sqrt[2p]{\ep}, 4\sqrt{\ep}).%
\end{equation*}
Since $\pi_0(f^T(\tilde{X}_{m_0})) = \pi_0(\tilde{X}_{m_0}) \geq 1 - \rho$, we thus obtain%
\begin{equation*}
  \pi_0\Bigl( \bigcup_{x\in F_n} B(f^T(x),n,2\sqrt[2p]{\ep},4\sqrt{\ep}) \Bigr) \geq 1 - \rho,%
\end{equation*}
showing that the set $f^T(F_n)$ is $(n,2\sqrt[2p]{\ep},4\sqrt{\ep},\rho)$-spanning. Consequently,%
\begin{equation*}
  h_{4\sqrt{\ep}}(f,\pi_0,2\sqrt[2p]{\ep},\rho) \leq \limsup_{n\rightarrow\infty}\frac{1}{n} \log |\MC|^{T+n} = C.%
\end{equation*}
Since $\ep>0$ was chosen arbitrarily, by \eqref{eq_limit} this implies%
\begin{equation*}
  C \geq \lim_{\ep\downarrow0}h_{4\sqrt{\ep}}(f,\pi_0,2\sqrt[2p]{\ep},\rho) = h_{\pi_0}(f),%
\end{equation*}
completing the proof.%
\end{proof}

\section{Discussion}\label{sec_discussion}

{\bf Hierarchy of estimation criteria.} Our results together with the results in the paper \cite{MPo} by Matveev and Pogromsky yield a hierarchy of estimation criteria for a deterministic nonlinear dynamical system%
\begin{equation*}
  x_{t+1} = f(x_t)%
\end{equation*}
with a continuous map (or a homeomorphism) $f:X \rightarrow X$ on a compact metric space $(X,d)$. These estimation criteria, ordered by increasing strength, are the following.%
\begin{enumerate}
\item[(1)] Asymptotic estimation in expectation (criterion (E3)):%
\begin{equation*}
  \limsup_{t\rightarrow\infty}\int_X \pi_0(\rmd x_0) d(f^t(x_0),\hat{x}_t(x_0))^p \leq \ep,\quad p>0%
\end{equation*}
with a Borel probability measure $\pi_0$ on $X$.%
\item[(2)] Asymptotic almost sure estimation (criterion (E2)):%
\begin{equation*}
  \pi_0(\{x_0\in X : \limsup_{t\rightarrow\infty} d(f^t(x_0),\hat{x}_t(x_0)) \leq \ep \}) = 1%
\end{equation*}
with a Borel probability measure $\pi_0$ on $X$.%
\item[(3)] Eventual almost sure estimation (criterion (E1)):%
\begin{equation*}
  \pi_0( \{ x_0 \in X : d(f^t(x_0),\hat{x}_t(x_0)) \leq \ep \} ) = 1 \mbox{\quad for all\ } t \geq T(\ep)%
\end{equation*}
with a Borel probability measure $\pi_0$ on $X$.%
\item[(4)] Deterministic estimation (called \emph{observability} in \cite{MPo}):%
\begin{equation*}
  d(x_t,\hat{x}_t) \leq \ep \mbox{\quad for all\ } t \geq T(\ep).%
\end{equation*}
\end{enumerate}
Finally, two even stronger estimation criteria are considered in \cite{MPo}, which do not fit exactly into the above hierarchy, since they are based on an initial error $d(x_0,\hat{x}_0) \leq \delta$, which is known to both coder and estimator, and can be much smaller than the desired accuracy $\ep$. In terms of this maximal initial error $\delta$, the criteria can be formulated as follows:%
\begin{enumerate}
\item[(5)] There are $\delta_*>0$ and $G \geq 1$ such that for all $\delta \in (0,\delta_*)$,%
\begin{equation*}
  d(x_t,\hat{x}_t) \leq G \delta \mbox{\quad for all\ } t \geq 0.%
\end{equation*}
This is called \emph{regular observability} in \cite{MPo}.%
\item[(6)] There are $\delta_*>0$, $G \geq 1$ and $g\in(0,1)$ such that for all $\delta \in (0,\delta_*)$,%
\begin{equation*}
  d(x_t,\hat{x}_t) \leq G \delta g^t \mbox{\quad for all\ } t \geq 0.%
\end{equation*}
This is called \emph{fine observability} in \cite{MPo}.%
\end{enumerate}
Let us denote by $C_0^{(i)}$, $1 \leq i \leq 6$, the smallest channel capacity for the above criteria. Then, modulo specific assumptions such as invertibility of $f$, ergodicity of $\pi_0$ and finiteness of entropies, our results together with those in \cite{MPo} essentially yield the following chain of inequalities:%
\begin{align*}
      h_{\pi_0}(f) &\stackrel{\mathrm{(a)}}{\leq} C_0^{(1)} \stackrel{\mathrm{(b)}}{\leq} h_{\pi_0}(f)\\
	                 &\stackrel{\mathrm{(c)}}{\leq} C_0^{(2)} \stackrel{\mathrm{(d)}}{\leq} h_{\tp}(f;\supp\pi_0)\\
	                 &\stackrel{\mathrm{(e)}}{\leq} C_0^{(3)} \stackrel{\mathrm{(f)}}{\leq} h_{\tp}(f;\supp\pi_0)\\
	 \stackrel{\mathrm{(g)}}{\leq} h_{\tp}(f) &\stackrel{\mathrm{(h)}}{\leq} C_0^{(4)} \stackrel{\mathrm{(i)}}{\leq} h_{\tp}(f)\\
	                 &\stackrel{\mathrm{(j)}}{\leq} C_0^{(5)}\\
								   &\stackrel{\mathrm{(k)}}{=} C_0^{(6)}.%
\end{align*}
For (h)--(k), see \cite[Thm.~8 and Lem.~14]{MPo}, and note that (d) can be replaced by $C_0^{(2)} \leq h_{\pi_0}(f)$ under the assumptions of Theorem \ref{thm_e2_ub2}. Recently, in \cite{MP2} the notion of \emph{restoration entropy} was proposed as an intrinsic quantity of a system which characterizes $C_0^{(5)} = C_0^{(6)}$. The results in \cite{MPo,MP2} show that this quantity can be strictly larger than $h_{\tp}(f)$.%

Except for compactness of $X$, continuity of $f$ and $f$-invariance of $\pi_0$, only the following inequalities rely on additional assumptions:%
\begin{enumerate}
\item[(a):] Invertibility of $f$, ergodicity of $\pi_0$ and periodic coding and estimation schemes using finite memory (see Theorem \ref{thm_e3_lb}).%
\item[(b):] Ergodicity of $\pi_0$ (see Theorem \ref{thm_e3_ub}).%
\item[(c):] Invertibility of $f$ (see Theorem \ref{thm_e2_lb}).%
\item[(e):] Invertibility of $f$ (see \cite[Thm.~3.3]{KYu}).%
\end{enumerate}

{\bf Relations to ergodic theory.} The results presented in this paper can be seen in analogy to certain fundamental results of ergodic theory about modeling a dynamical system by a symbolic system, i.e., a subshift of a full shift on $n$ symbols. For instance, Krieger's generator theorem roughly says that an ergodic system can be modeled as a subshift over an alphabet of $l > 2^h$ symbols, where $h$ is the entropy of the system. In our context, the modeling of the system is accomplished by the coder which generates an infinite sequence of symbols $(q_0,q_1,q_2,\ldots)$ over the alphabet $\MC$ for any initial state $x_0$. However, one should be careful with this analogy, because we are allowing that the coder is defined by a time-varying function, which not necessarily results in a commutative diagram relating the action of $f$ (or some iterate of $f$) and the action of a shift operator. For topological dynamical systems, the theory of entropy structures and symbolic extensions answers the question to which extent a system can be modeled by a symbolic system (under preservation of the topological structure), cf.~\cite{Downarowicz} for an excellent overview of this theory.%

{\bf Open questions.} Since in most of our results we had to use conditions more specific than $f$-invariance of $\pi_0$, the question remains to which extent those results also hold without these specific assumptions. Particularly strong assumptions have been used in Theorem \ref{thm_e2_ub2} and in Theorem \ref{thm_e3_lb}. We have not been able to find examples where these assumptions are violated and the results do not hold, so this is definitely a topic for future research. Finally, we remark that the assumption of invertibility of $f$ was only used to obtain lower bounds for each of the objectives (E1)--(E3), while we did not use this assumption to estimate $C_0$ from above. Hence, another open question is whether the lower bounds also hold when $f$ is not invertible.%

\section{Corollaries and examples}\label{sec_corex}

For differentiable maps, smooth ergodic theory provides a relation between the metric entropy and the Lyapunov exponents of a map, whose existence follows from Oseledec's Theorem (see Theorem \ref{thm_met} in the appendix). In particular, we have the following corollary of our main theorems.%

\begin{corollary}
Assume that $f$ is a $C^1$-diffeomorphism on a compact Riemannian manifold preserving an ergodic measure $\pi_0$. Then, in the setting of Theorem \ref{thm_e2_ub2} or Theorem \ref{thm_e3_ub},%
\begin{equation*}
  C_0 \leq \log(1 + \lfloor 2^{\int \pi_0(\rmd x)\lambda^+(x)} \rfloor),%
\end{equation*}
where $\lambda^+(x)$ is a shortcut for the sum of the positive Lyapunov exponents at $x$. If $f$ is a $C^2$-diffeomorphism and $\pi_0$ an SRB measure (not necessarily ergodic), then, in the setting of Theorem \ref{thm_e2_lb} or Theorem \ref{thm_e3_lb},%
\begin{equation*}
  C_0 \geq \int \pi_0(\rmd x) \lambda^+(x).%
\end{equation*}
\end{corollary}

\begin{proof}
The first statement follows from the Margulis-Ruelle Inequality for the metric entropy of a $C^1$-diffeomorphism (Theorem \ref{thm_ruelle_pesin} in the appendix). The second statement follows from the characterization of SRB measures for $C^2$-diffeomorphisms in terms of metric entropy (Theorem \ref{thm_ledrappier_young} in the appendix).%
\end{proof}

In the following, we demonstrate the contents of our results by applying them to some well-studied dynamical systems with chaotic behavior (and thus positive entropy).%

\begin{example}
Consider the map%
\begin{equation*}
  f(\theta,x,y) = (2\theta,x/4 + \cos(2\pi\theta)/2,y/4 + \sin(2\pi\theta)/2),%
\end{equation*}
regarded as a map on $X := S^1 \tm D^2$, where $S^1$ is the unit circle and $D^2$ the unit disk, respectively. It is known that $f$ has an Axiom A attractor $\Lambda$, which is called the \emph{solenoid attractor}. In this case, there exists a unique SRB measure $\pi_0$ supported on $\Lambda$. It is known that $f$ is topologically transitive on $\Lambda$, implying that $\pi_0$ is ergodic (cf.~\cite[Sec.~7.7]{Rob}). The metric entropy $h_{\pi_0}(f)$ is known to be $\log 2$, cf.~\cite[Ex.~6.3]{Fro}. Hence, in this case%
\begin{equation*}
  C_0 = \log 2 = 1,%
\end{equation*}
i.e., the channel must support a transmission of at least one bit at each time instant in order to make the state estimation objectives (E2) and (E3) achievable.%
\end{example}

\begin{example}
We consider the map%
\begin{equation*}
  f(x,y) = (5 - 0.3y - x^2,x),\quad f:\R^2 \rightarrow \R^2.%
\end{equation*}
The nonwandering set of $f$, i.e., the set of all points $(x,y)$ so that for every neighborhood $U$ of $(x,y)$ there is $n\geq 1$ with $f^n(U) \cap U \neq \emptyset$, can be shown to be uniformly hyperbolic, but it is not an attractor. Hence, there exists no SRB measure on this set. However, we can take $\pi_0$ to be the unique equilibrium state of the potential $\varphi(x) := -\log|\rmD f(x)_{|E^u_x}|$ ($E^u_x$ being the unstable subspace at $x$), i.e., the unique probability measure $\pi_0$ satisfying%
\begin{equation*}
  P_{\tp}(f;\varphi) = h_{\pi_0}(f) + \int \varphi \rmd \pi_0,%
\end{equation*}
where $P_{\tp}(\cdot)$ is the topological pressure.%

The restriction of $f$ to its nonwandering set is an Axiom A system. A numerical approximation of the metric entropy $h_{\pi_0}(f)$ is given in \cite[Ex.~6.4]{Fro}, namely%
\begin{equation*}
  h_{\pi_0}(f) \approx 0.655.%
\end{equation*}
\end{example}

\begin{example}
Consider a volume-preserving $C^2$-Anosov diffeomorphism $f:\T^n \rightarrow \T^n$ of the $n$-dimensional torus. It is known that the volume measure $m$ is ergodic for such $f$. In this case,%
\begin{align*}
  h_m(f) = \int \log J^u f \rmd m,\quad &h_{\tp}(f) = h_{\tp}(f;\supp m) \\
&\quad = \lim_{n\rightarrow\infty}\frac{1}{n}\log \int J^u f^n \rmd m,%
\end{align*}
where $J^u f$ stands for the unstable determinant of $\rmD f$, i.e., $J^uf(x) = |\det\rmD f(x)_{|E^u_x}|$ with $E^u_x$ being the unstable subspace at $x$. Observe that, in general, $h_m(f) \neq h_{\tp}(f)$. In fact, by \cite[Cor.~20.4.5]{KHa}, in the case $n=2$ the equality $h_m(f) = h_{\tp}(f)$ implies that $f$ is $\CC^1$-conjugate to a linear automorphism $g:\T^2 \rightarrow \T^2$. Since $\CC^1$-conjugacy implies that the eigenvalues of the corresponding linearizations along any periodic orbits coincide, the existence of a $\CC^1$-conjugacy is a very strict condition. Hence, most area-preserving Anosov diffeomorphisms on $\T^2$ satisfy the strict inequality $h_m(f) < h_{\tp}(f)$.%
\end{example}

\begin{example}
An example for an area-preserving Anosov diffeomorphism on $\T^2 = \R^2/\Z^2$ is Arnold's Cat Map, which is the group automorphism induced by the linear map on $\R^2$ associated with the integer matrix%
\begin{equation*}
  A = \left[\begin{array}{cc} 2 & 1 \\ 1 & 1 \end{array}\right].%
\end{equation*}
The fact that this matrix has determinant $1$ implies that the induced map $f_A:\T^2 \rightarrow \T^2$, $x+\Z^2 \mapsto Ax + \Z^2$, is area-preserving and invertible. The Anosov property follows from the fact that $A$ has two distinct real eigenvalues $\lambda_1$ and $\lambda_2$ with $|\lambda_1| > 1$ and $|\lambda_2| < 1$, namely%
\begin{equation*}
  \lambda_1 = -\frac{3}{2} - \frac{1}{2}\sqrt{5} \mbox{\quad and\quad} \lambda_2 = -\frac{3}{2} + \frac{1}{2}\sqrt{5}.%
\end{equation*}
At each point in $\T^2$, the Lyapunov exponents of $f_A$ exist and their values are $\log|\lambda_1|$ and $\log|\lambda_2|$. Pesin's entropy formula (see Theorem \ref{thm_ruelle_pesin} in the appendix) thus implies%
\begin{equation*}
  h_m(f_A) = \log|\lambda_1| = \log\bigg|\frac{3}{2} + \frac{1}{2}\sqrt{5}\bigg|.%
\end{equation*}
The formula for the topological entropy given in the preceding example shows that $h_{\tp}(f_A) = h_m(f_A)$.%
\end{example}

Finally, we present an example taken from \cite{VZh} in which the assumptions of Theorem \ref{thm_e2_ub2} are satisfied.%

\begin{example}
Let $f:\rmS^1 \rightarrow \rmS^1$ be an expanding circle map of degree $\geq 3$. (Observe that we did not use invertibility of $f$ in Theorem \ref{thm_e2_ub2}.) Write $\rmS^1 = [0,1]/\!\!\sim$, where the equivalence relation $\sim$ identifies the extremal points of the interval. Assume that there exists an open interval $I \subset (0,1)$ so that $f_{|I}$ is injective and $f(I) = (0,1)$. Then the set%
\begin{equation*}
  \Lambda := \bigcap_{n \geq 0}f^{-n}(\rmS^1 \backslash I)%
\end{equation*}
is an $f$-invariant Cantor set and $f_{|\Lambda}$ is expanding. By \cite[Ex.~5.2]{VZh}, we find ergodic measures $\mu$ of $f_{|\Lambda}$ and partitions $\PC$ of $\Lambda$ (of arbitrarily small diameter) such that there exists $\delta_*>0$ with%
\begin{equation*}
  \limsup_{n\rightarrow\infty}\frac{1}{n} \log\mu\bigl(\bigl\{ x\in\Lambda : \bigl|-\frac{1}{n}\log \mu(\PC^n(x)) - h\bigr| \geq \delta\bigr\}\bigr) < 0%
\end{equation*}
for all $\delta \in (0,\delta_*)$, where $h = h_{\mu}(f_{|\Lambda}) = h_{\mu}(f_{|\Lambda};\PC)$. For those measures, Theorem \ref{thm_e2_ub2} yields%
\begin{equation*}
  C_0 \leq \log(1 + \lfloor 2^{h_{\mu}(f_{|\Lambda})} \rfloor),%
\end{equation*}
where $C_0$ is the smallest channel capacity above which (E2) can be achieved for every $\ep>0$. The measures that satisfy the above are obtained from so-called Gibbs measures of the one-sided full shift on two symbols, which is topologically conjugate to $f_{|\Lambda}$.%
\end{example}

\section{Appendix}\label{sec_appendix}

In this appendix, we review some classial results of ergodic theory that are used in the paper. We start with the Shannon-McMillan-Breiman Theorem, cf.~\cite{GrayIT}.%

\begin{theorem}\label{thm_smb}
Let $(X,\FC,\mu,f)$ be an ergodic measure-preserving dynamical system, i.e., $(X,\FC,\mu)$ is a probability space and $\mu$ is ergodic w.r.t.~the map $f:X\rightarrow X$. Let $\PC$ be a finite measurable partition of $X$. Then, for almost every $x\in X$, as $n\rightarrow\infty$,%
\begin{equation*}
  -\frac{1}{n}\log\mu(\PC^n(x)) \rightarrow h_{\mu}(f;\PC),%
\end{equation*}
where $\PC^n(x)$ is the element of $\PC^n$ containing $x$.%
\end{theorem}

If $\PC = \{P_1,\ldots,P_r\}$, for $n\in\N$ and $x\in X$, let $i^n(x) = (i^n_0(x),\ldots,i^n_{n-1}(x))$ be the element of $\{1,\ldots,r\}^n$ with $f^j(x) \in P_{i^n_j(x)}$ for $0 \leq j < n$. We call $i^n(x)$ the \emph{length-$n$ itinerary of $x$ w.r.t.~$\PC$}. Putting $h := h_{\mu}(f;\PC)$, we also define for each $\ep>0$ and $n\in\N$ the set%
\begin{align*}
 & \Sigma_{n,\ep} := \bigl\{ a \in \{1,\ldots,r\}^n : 2^{-n(h+\ep)} \leq \\
 & \quad \quad \quad \quad \mu(\{x:i^n(x)=a\}) \leq 2^{-n(h-\ep)} \bigr\}.%
\end{align*}

As an immediate corollary of the Shannon-McMillan-Breiman Theorem we obtain the following.%

\begin{corollary}\label{cor_smb}
Under the assumptions of Theorem \ref{thm_smb}, it holds that%
\begin{equation}\label{eq_tsmeas}
  \mu(\{x\in X : \exists m\in\N \mbox{ with } i^n(x) \in \Sigma_{n,\ep},\ \forall n \geq m\}) = 1%
\end{equation}
and%
\begin{equation}\label{eq_tscard}
  |\Sigma_{n,\ep}| \leq 2^{n(h+\ep)}.%
\end{equation}
\end{corollary}

\begin{proof}
The identity \eqref{eq_tsmeas} follows from the observation that $\mu(\PC^n(x)) = \mu(\{y\in X : i^n(y) = i^n(x)\})$. From the estimates%
\begin{align*}
&  1 \geq \mu(\{x : i^n(x) \in \Sigma_{n,\ep}\}) \\
&  = \sum_{a\in\Sigma_{n,\ep}}\mu(\{x:i^n(x)=a\}) \geq |\Sigma_{n,\ep}|\cdot 2^{-n(h+\ep)}%
\end{align*}
the inequality \eqref{eq_tscard} immediately follows.%
\end{proof}

In the ergodic theory of smooth dynamical systems, the most fundamental result is Oseledec's Theorem, also known as the Multiplicative Ergodic Theorem, see \cite{Arn}.%

\begin{theorem}\label{thm_met}
Let $(\Omega,\FC,\mu)$ be a probability space and $\theta:\Omega \rightarrow \Omega$ a measurable map, preserving $\mu$. Let $T:\Omega \rightarrow \R^{d\tm d}$ be a measurable map such that $\max\{0,\log\|T(\cdot)\|\} \in L^1(\Omega,\mu)$ and write%
\begin{equation*}
  T_x^n := T(\theta^{n-1}(x)) \cdots T(\theta(x))T(x).%
\end{equation*}
Then there is $\tilde{\Omega} \subset \Omega$ with $\mu(\tilde{\Omega}) = 1$ so that for all $x\in\tilde{\Omega}$ the following holds:%
\begin{equation*}
  \lim_{n\rightarrow\infty}\left[(T_x^n)^*(T_x^n)\right]^{1/(2n)} =: \Lambda_x%
\end{equation*}
exists and, moreover, if $\exp\lambda_x^{(1)} < \cdots < \exp\lambda_x^{s(x)}$ denote the eigenvalues of $\Lambda_x$ and $U_x^{(1)},\ldots,U_x^{s(x)}$ the associated eigenspaces, then%
\begin{equation*}
  \lim_{n\rightarrow\infty}\frac{1}{n}\log\|T_x^nv\| = \lambda_x^{(r)} \mbox{ if } v \in V_x^{(r)}\backslash V_x^{(r-1)},%
\end{equation*}
where $V_x^{(r)} = U_x^{(1)} + \cdots + U_x^{(r)}$ and $r = 1,\ldots,s(x)$.%
\end{theorem}

The numbers $\lambda_x^{(i)}$ are called \emph{($\mu$-)Lyapunov exponents}. We apply the theorem to the situation, when $\theta$ is a diffeomorphism on a compact Riemannian manifold and $T(x)$ is the derivative of $\theta$ at $x$. In this situation, the Lyapunov exponents measure the rates at which nearby orbits diverge (or converge). Despite the fact that this is a purely geometric concept, the Lyapunov exponents are closely related to the metric entropy. Fundamental results here are the Margulis-Ruelle inequality \cite{Rue} and Pesin's entropy formula \cite{Pes}, which we summarize in the following theorem.%

\begin{theorem}\label{thm_ruelle_pesin}
Let $M$ be a compact Riemannian manifold and $f:M\rightarrow M$ a $C^1$-diffeomorphism preserving a Borel probability measure $\mu$. Then%
\begin{equation*}
  h_{\mu}(f) \leq \int \mu(\rmd x)\sum_i\max\{0,\dim U_x^{(i)} \cdot \lambda_x^{(i)}\}%
\end{equation*}
with summation over all Lyapunov exponents at $x$. If $f$ is a $C^2$-diffeomorphism and $\mu$ is absolutely continuous w.r.t.~Riemannian volume, then equality holds. If $\mu$ is ergodic, then the Lyapunov exponents are constant almost everywhere, and hence the integration can be omitted.%
\end{theorem}

If a diffeomorphism does not preserve a measure absolutely continuous with respect to volume, it might still preserve a measure that is well-behaved with respect to volume in the sense that the measure determines the behavior of all trajectories with initial values in a set of positive volume. This is made precise in the definition of an SRB (Sinai-Ruelle-Bowen) measure. We will not give the technical definition here, but just mention the following result of Ledrappier and Young \cite{LYo}, which can in fact be regarded as one possible definition of an SRB measure.%

\begin{theorem}\label{thm_ledrappier_young}
Let $f:M\rightarrow M$ be a $C^2$-diffeomorphism of a compact Riemannian manifold, preserving a Borel probability measure $\mu$. Then Pesin's entropy formula for $h_{\mu}(f)$ holds if and only if $\mu$ is an SRB measure.%
\end{theorem}

Finally, we present the proof of Lemma \ref{lem_newhouse}.%

\begin{proof}
Throughout the proof, we use the notation $\RC|\tilde{X} = \{R\cap\tilde{X} : R \in \RC\}$ for any partition $\RC$ of $X$. Given the partition $\PC$, choose $\delta$ so that%
\begin{equation*}
  0 < \delta < \frac{1}{2}\min_{1 \leq i < j \leq s}\min_{(x,y) \in P_i \tm P_j}d(x,y).%
\end{equation*}
Given $\rho>0$, choose $\alpha \in (0,1)$ satisfying%
\begin{equation*}
  \alpha < \frac{\rho}{2\log|\PC|} = \frac{\rho}{2\log(s+1)}.%
\end{equation*}
Now let $\tilde{X} \subset X$ be an arbitrary measurable set with $\pi_0(\tilde{X}) \geq 1 - \alpha$ and consider the partition $\QC := \{\tilde{X},X\backslash\tilde{X}\}$. Fix $N\in\N$ with $(\log 2)/N < \rho/2$. Let $E \subset \tilde{X}$ be a maximal $(k,\delta/2)$-separated subset for the map $f^N$. Maximality guarantees that for each $z \in \tilde{X}$ there is $\varphi(z) \in E$ such that $d(f^{jN}(z),f^{jN}(\varphi(z))) \leq \delta/2$ for $0 \leq j < k$. Using elementary properties of conditional entropy, we obtain%
\begin{align*}
&  H_{\pi_0}\Bigl(\bigvee_{i=0}^{k-1}f^{-iN}\PC\Bigr) \\
  &\leq H_{\pi_0}(\QC) + H_{\pi_0}\Bigl(\bigvee_{i=0}^{k-1}f^{-iN}\PC|\QC\Bigr)\\
	&\leq \log 2 + \pi_0(\tilde{X}) \sum_{P \in \bigvee_{i=0}^{k-1}f^{-iN}\PC} - \frac{\pi_0(P \cap \tilde{X})}{\pi_0(\tilde{X})}\log \frac{\pi_0(P \cap \tilde{X})}{\pi_0(\tilde{X})}\\
	&\qquad\qquad + \pi_0(X\backslash\tilde{X})(k\log|\PC|)\\
	&\leq \log 2 + \log\Bigl|\bigvee_{i=0}^{k-1}f^{-iN}\PC|\tilde{X}\Bigr| + \frac{\rho}{2}k.%
\end{align*}
For each $P \in \bigvee_{i=0}^{k-1}f^{-iN}\PC|\tilde{X}$ let $x_P \in P$ and consider $\varphi(x_P) \in E$. Since each ball $B_{\delta}(z)$ intersects at most two elements of $\PC$, we have $|\varphi^{-1}(z)| \leq 2^k$, implying%
\begin{align*}
  \log\bigl|\bigvee_{i=0}^{k-1}f^{-iN}\PC|\tilde{X}\bigr| &\leq k\log 2 + \log|E|\\
	&\leq k\log 2 + \log r_{\sep}(kN,\delta/2;\tilde{X},f),%
\end{align*}
using that $E$ is a $(kN,\delta/2)$-separated (though not necessarily maximal) subset of $\tilde{X}$ for the map $f$. This yields%
\begin{align*}
  H_{\pi_0}\Bigl(\bigvee_{i=0}^{k-1}f^{-iN}\PC\Bigr) &\leq \log 2 + k\log 2 \\
	&\qquad+ \log r_{\sep}(kN,\delta/2;\tilde{X},f) + \frac{\rho}{2}k.%
\end{align*}
Dividing by $kN$ and letting $k\rightarrow\infty$ leads to $(1/N)h_{\pi_0}(f^N;\PC) \leq h(f;\tilde{X},\delta/2) + \rho$, completing the proof.%
\end{proof}

\section*{Acknowledgements}%

We would like to express our gratitude to Anthony Quas who gave us a very helpful hint leading to Theorem \ref{thm_e2_ub2}, to Katrin Gelfert who helped us to gain a better understanding of Gibbs measures. We also thank the reviewers and the AE for their thoughtful comments leading to an improvement of the paper.%

\end{document}